\documentclass[11pt,reqno]{amsart}

\usepackage{graphicx}

\usepackage{amssymb,amsmath}

\usepackage{tikz,caption}
\usepackage{enumitem}

\usepackage{verbatim}
\usetikzlibrary{arrows,shapes}
\usetikzlibrary{positioning}

\def\kb{{\mathbf k}}
\def\Xk{{\mathcal X}}
\def\Hk{{\mathcal H}}
\def\Fk{{\mathcal F}}
\def\Gk{{\mathcal G}}
\def\Ck{{\mathcal C}}
\def\Pk{{\mathcal P}}
\def\Cf{{\mathfrak C}}
\def\RR{{\mathbb R}}
\def\NN{{\mathbb N}}

\def\lt{{\rm lt}}
\def\cls{{\rm cls}}
\def\beg{{\rm beg}}
\def\endfin{{\rm end}}
\def\sgn{{\rm sgn}}
\def\set{{\rm set}}
\def\ch{{\rm ch}}
\def\pd{{\rm pd}}
\def\md{{\rm md}}
\def\lcm{{\rm lcm}}
\def\gr{{\rm gr}}

\newtheorem{definition}{Definition}[section]

\newtheorem{proposition}{Proposition}[section]

\newtheorem{theorem}{Theorem}[section]

\newtheorem{example}{Example}[section]

\newtheorem{remark}{Remark}[section]

\begin{document}

\title[Cellular structure of the Pommaret-Seiler Resolution]{Cellular structure of the Pommaret-Seiler resolution for quasi-stable ideals}
\author{Rodrigo Iglesias \and Eduardo S\'aenz-de-Cabez\'on}

\maketitle

\begin{abstract}
We prove that the Pommaret-Seiler resolution for quasi-stable ideals is cellular and give a cellular structure for it. This shows that this resolution is a generalization of the well known Eliahou-Kervaire resolution for stable ideals in a deeper sense. We also prove that the Pommaret-Seiler resolution can be reduced to the minimal one via Discrete Morse Theory and provide a constructive algorithm to perform this reduction.
\end{abstract}

\section{Introduction}
Involutive bases lie at the interplay of the algebraic theory of differential equations and Gr\"obner bases theory. The algebraic approach to differential systems of equations by Riquier \cite{R10} and Janet \cite{J24} in the first decades of the twentieth century included, among others, the concept of multiplicative variables and involutive division. In fact, the concept of involution is central to the formal analysis of non-normal systems of differential equations. The geometric formalism of the formal theory of differential systems leads in a natural way to strong connections to commutative algebra. For instance, the notion of the degree of involution of a differential equation is equivalent to the Castelnuovo-Mumford regularity of a certain polynomial module associated to it. Milestone works in this direction include those by Spencer \cite{S69}, Pommaret \cite{P78} and Seiler \cite{S10} among others. The connection to  Gr\"obner bases theory, was established first by Wu \cite{W91} and by Gerdt et al. \cite{GB98a,GB98b,G99}  in whose work involutive bases were introduced as a special kind of Gr\"obner bases with additional combinatorial properties. The main idea behind involutive bases is to assign  a subset of the variables to each generator in a basis  of the ideal or module under consideration. This is called its set of {\em multiplicative variables}. This assignment is called an involutive division because one only allows multiplication of each generator by polynomials in its multiplicative variables. Different rules of assignment of multiplicative variables define different involutive divisions like  the Thomas, Janet or Pommaret divisions, to name the main ones. The multiplication restrictions that define involutive divisions lead to extra combinatorial properties beyond those of Gr\"obner bases. Furthermore, involutive bases have been generalized to non-commutative contexts \cite{E06,S10,CM20}.

In his two papers \cite{S09a,S09b}, Seiler performs an extensive study of involutive bases in the general context of algebras of solvable type, with an emphasis on Pommaret bases. When focusing on structural properties of ideals or modules that can be determined by the use of involutive bases, Pommaret bases are particularly useful. They may be applied for determining the Krull and projective dimensions, and the depth of a polynomial module. Seiler provides for instance brief proofs of Hironaka’s criterion for Cohen-Macaulay modules and
of the graded form of the Auslander-Buchsbaum formula, using Pommaret bases. In syzygy theory, Pommaret bases can be used to construct a free resolution which is generically minimal for componentwise linear modules.

In the case of polynomial ideals, finite Pommaret bases always exist, up to a generic change of coordinates. This is of course not the case for monomial ideals, since under a generic change of coordinates monomial ideals are no longer monomial. This raises the question of characterizing those monomial ideals for which a finite Pommaret basis exists, which are called {\em quasi-stable} ideals. 
For these, the syzygy complex given by the Pommaret basis has a differential algebra structure, and an explicit formula
for the differential is obtained. If the ideal is stable, this resolution, called the {\em Pommaret-Seiler resolution}, is minimal. Several important invariants such as the Castelnuovo-Mumford regularity, projective dimension or extremal Betti numbers can be read directly off the Pommaret basis via this resolution.


Our work is focused on quasi-stable ideals and free resolutions of them, bringing the theory of involutive bases into the combinatorial approach to commutative algebra, in which monomial ideals are one of the central concepts under study. Free resolutions are the main object in the homological approach to ideals and modules in commutative algebra, and the computation of a minimal one is very demanding, even in the case of monomial ideals, being one of the main problems in the area. Two main strategies have been used in the literature to study free resolutions of monomial ideals. One is the study of explicit minimal free resolutions for particular families of monomial ideals, and the other is the development of general procedures to obtain free resolutions that (although possibly not minimal) provide homological information on the ideal.
The main result of the first approach is the explicit description of the minimal free resolution of stable ideals given by Eliahou and Kervaire \cite{EK90}, and the squarefree versions of this given by Herzog and coauthors in \cite{AHH98,GHP02}. The main construction within the second approach is the Taylor resolution \cite{T66}, which is a combinatorial explicit resolution that is usually not minimal. A more compact resolution derived from this was given by Lyubeznik \cite{L88}. Taylor and Lyubeznik resolutions are two instances of {\em cellular resolutions} \cite{BS98}, which form another general strategy for producing monomial resolutions. This takes advantage of the combinatorial nature of monomial ideals to encode their resolutions by means of cellular complexes. As a complement to this line of research, several techniques for minimizing a given resolution have been developed, in particular those based on  Discrete Vector Fields (a subject initiated by R. Forman in \cite{F98}, satellite of his Discrete Morse Theory), see \cite{AFG20} for a recent example. Another general technique for constructing monomial resolutions is the iterated mapping cone \cite{HT02}, which allows the construction of (non-minimal) free resolutions for arbitrary polynomial and monomial ideals. For instance, the already mentioned Taylor and Eliahou-Kervaire resolutions are examples of iterated mapping cones.

The connection of involutive bases and the theory of monomial resolutions was made by Seiler in \cite{S09b}. Pommaret-Seiler resolution is a (non-minimal) explicit free resolution for quasi-stable ideals. It is a direct generalization of the Eliahou and Kervaire resolution for stable ideals and shares several properties with it, like the fact that both can be seen as iterated mapping cones, see \cite{AFSS15}. Another important property of the Eliahou-Kervaire resolution is that it is a cellular resolution \cite{M10}, hence a natural question is whether the Pommaret-Seiler resolution is also cellular. We give an affirmative answer in Theorem \ref{th:cellular} and also give an explicit cellular structure for the Pommaret-Seiler resolution. These results are based on previous work in \cite{M10} on stable ideals and in particular on the results in \cite{DM14} which show the cellular nature of the mapping cone resolution of ideals with linear quotients. Our approach is to extend their work to quasi-stable ideals using Pommaret bases as generating sets. Furthermore, we use the algebraic version of Discrete Morse Theory in order to reduce the Pommaret-Seiler resolution of a quasi-stable ideal to obtain its minimal free resolution. These contributions play a part in the development of the combinatorial aspects of involutive bases and their connection to monomial ideal theory. Some recent works on other aspects of this connection are \cite{C21,BC22}.

The outline of the paper is as follows. In Section \ref{sec:PS-EK} we introduce the Pommaret-Seiler resolution and its relation to the Eliahou-Kervaire one. In Section \ref{sec:cellular} we prove that the Pommaret-Seiler resolution is cellular and give an actual cellular structure for it. Finally, in Section \ref{sec:reduction} we provide a reduction procedure to obtain minimal free resolutions of quasi-stable ideals.

\section{Pommaret-Seiler and Eliahou-Kervaire resolutions} \label{sec:PS-EK}
Involutive divisions (in particular the Pommaret division) can be defined in very general contexts. We however restrict here to the setting of monomial ideals in the commutative polynomial ring on $n$ variables. The definitions that follow are therefore limited to monomial ideals, but more general versions can be found in the literature; we refer the interested reader to \cite{S09a} or \cite{GB98a} for instance.

\subsection{Pommaret bases and quasi-stable ideals}
Let $R=\kb[x_1,\dots,x_n]$ be the polynomial ring on $n$ variables over a field $\kb$ and denote by $\Xk$ the set of variables $\{x_1,\dots,x_n\}$ \footnote{Observe that in the notation of \cite{S09a,S09b}, which we follow here, the ordering of the variables is reversed with respect to the traditional one used for instance in \cite{EK90}.}. Let $\mu=(\mu_1,\dots,\mu_n)\in \mathbb{N}^n$. For a monomial $x^\mu\in R$ 
we say that the {\em class} of $\mu$ or $x^\mu$,
denoted by $\cls(\mu)=\cls(x^\mu)$,
 is equal to $\min\{ i\vert \mu_i\neq 0\}$. We say that the {\em multiplicative variables} of $x^\mu$
are $\Xk_P(x^\mu)=
\{x_1,\dots,x_{\cls (\mu)}\}$, and we denote by $\overline{\Xk}_P(x^\mu)$ the set of {\em non-multiplicative variables} of $x^\mu$, given by $\Xk\setminus \Xk_P(x^\mu)$. We say that $x^\mu$ is an {\em involutive divisor} of $x^\nu$ if $x^\mu\vert x^\nu$ and $x^{\nu-\mu}\in \kb[\Xk_P({x^\mu})]$.

\begin{definition}
Let $\Hk$ be a finite collection of monomials, $\Hk\subseteq R$. We say that $\Hk$ is a {\em Pommaret basis} of the monomial ideal $I=\langle \Hk \rangle$ if $I=\bigoplus_{h\in\Hk} h\cdot\kb[\Xk_P(h)]$.

We call a monomial ideal $I$ {\em quasi-stable}, if it possesses a finite monomial Pommaret basis.

\end{definition}

\begin{definition}
For every $h_i\in \Hk$, the Pommaret basis of a quasi-stable ideal $I$, we define the {\em involutive monomial cone} $\Ck(h_i)$ of $h_i$ with respect to the Pommaret division as the set of all monomials in $h_i\cdot\kb[\Xk_P(h_i)]$. 
\end{definition}

\begin{remark}
A quick glance at the polynomial setting: a finite polynomial set $\Hk$ is a Pommaret basis of the polynomial ideal $I=\langle \Hk\rangle$ for the monomial ordering $\prec$, if all the elements of $\Hk$ possess distinct leading terms and these terms form a Pommaret basis of the leading ideal $\lt_\prec(I)$.
\end{remark}

\begin{example}\label{ex:quasi-stable-ideal}
Consider the monomial ideal $I\subseteq\kb[x_1,x_2]$ given by $I=\langle x_1^2, x_2^3\rangle$. The set of multiplicative variables with respect to the Pommaret division for each of the generators are $\Xk_P(x_1^2)=\{x_1\}$ and $\Xk_P(x_2^3)=\{x_1,x_2\}$. The monomial $x_1^2x_2$ is in $I$ but it is not in the involutive cone of any of its minimal generators, hence the minimal generating set of $I$ is not a Pommaret basis of it. The involutive basis of $I$ with respect to the Pommaret division is given by $\Hk=\{x_1^2,x_1^2x_2,x_1^2x_2^2,x_2^3\}$, therefore $I$ is a quasi-stable ideal. Figure \ref{fig:Pbasis} shows on one side the minimal generating set  of $I$ and their usual (overlapping) multiplicative cone of each of its elements, and on the other side the Pommaret basis of $I$ and the involutive (non overlapping) cone of each of its elements.

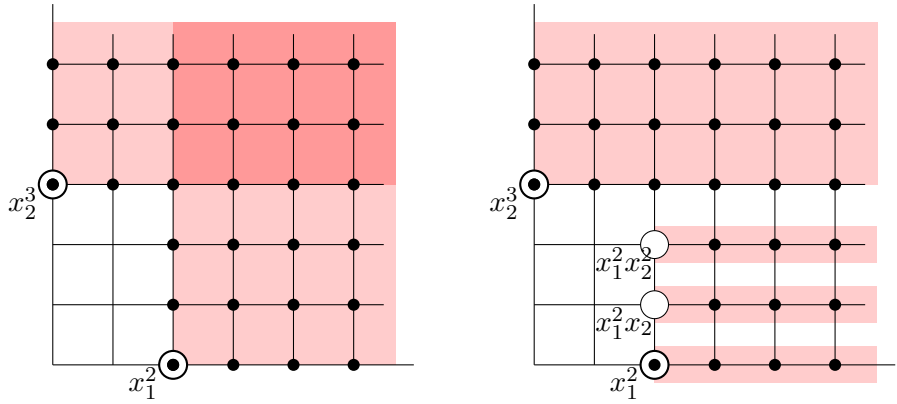
\begin{figure}[t]
\begin{center}
\begin{tikzpicture}[scale=0.8]
\tikzset{dot/.style={draw=black,thick,fill=white,circle}}
\tikzset{wdot/.style={draw=black,fill=white,circle}}

\fill[red!20!white] (0,3) rectangle (5.7,5.7);
\fill[red!20!white] (2,0) rectangle (5.7,5.7);
\fill[red!40!white] (2,3) rectangle (5.7,5.7);


\foreach\l[count=\y] in {$x_2$, $x_2^2$,$x_2^3$, $x_2^4$,$x_2^5$}
{
\draw (0,\y) -- (5.5,\y);
}

\draw(0,0)--(6,0);
\draw(0,0)--(0,6);
\foreach \l[count=\x] in {$x_1$, $x_1^2$,$x_1^3$, $x_1^4$,$x_1^5$}
{
\draw (\x,0) -- (\x,5.5);
}
\node at (1.5,-0.3){$x_1^2$};
\node at (-0.5,2.7){$x_2^3$};

\node[dot] at (0,3){};
\node[dot] at (2,0){};
\fill(0,3) circle (.1cm);
\fill(2,0) circle (.1cm);


\foreach \i in {2,...,5}{
    \foreach \j in {0,...,2}{
    {
     \fill(\i,\j) circle (.1cm);
     }
    }
   }
   
 \foreach \i in {0,...,5}{
    \foreach \j in {3,...,5}{
    {
     \fill(\i,\j) circle (.1cm);
     }
    }
   }

\fill[red!20!white] (8,3) rectangle (13.7,5.7);
\fill[red!20!white] (10,-.3) rectangle (13.7,.3);
\fill[red!20!white] (10,.7) rectangle (13.7,1.3);
\fill[red!20!white] (10,1.7) rectangle (13.7,2.3);

\foreach\l[count=\y] in {$x_2$, $x_2^2$,$x_2^3$, $x_2^4$,$x_2^5$}
{
\draw (8,\y) -- (13.5,\y);
}

\draw(8,0)--(14,0);
\draw(8,0)--(8,6);
\foreach \l[count=\x] in {$x_1$, $x_1^2$,$x_1^3$, $x_1^4$,$x_1^5$}
{
\draw (\x+8,0) -- (\x+8,5.5);
}
\node at (9.5,-0.3){$x_1^2$};
\node at (7.5,2.7){$x_2^3$};

\node[dot] at (8,3){};
\node[dot] at (10,0){};
\fill(8,3) circle (.1cm);
\fill(10,0) circle (.1cm);

\node[wdot] at (10,1){};
\node[wdot] at (10,2){};
\node at (9.5,.7){$x_1^2x_2$};
\node at (9.5,1.7){$x_1^2x_2^2$};

\foreach \i in {11,...,13}{
    \foreach \j in {0,...,2}{
    {
     \fill(\i,\j) circle (.1cm);
     }
    }
   }
   
 \foreach \i in {8,...,13}{
    \foreach \j in {3,...,5}{
    {
     \fill(\i,\j) circle (.1cm);
     }
    }
   }
\end{tikzpicture}
\caption{Minimal generating set (left) and Pommaret basis (right) for $I=\langle x_1^2, x_2^3\rangle$. In the left diagram the usual cones of each of the generators overlap on all their common multiples, while in the right diagram, the involutive cones do not overlap, yielding a disjoint decomposition of the set of monomials in $I$.}\label{fig:Pbasis}
\end{center}
\end{figure}

\end{example}

Quasi-stable monomial ideals can be characterized in several ways which are independent of the theory of Pommaret bases, see \cite[Proposition 5.3.4]{S10}. They have appeared in the literature under the names of {\em ideals of nested type} \cite{BG06}, {\em weakly stable ideals} \cite{CS05} or {\em ideals of Borel-type} \cite{HPV03}. The name quasi-stable is due to the fact that these ideals are a generalization of the important class of {\em stable ideals}. The property of being quasi-stable is preserved by operations like sum, intersection, product and colon, see \cite{C09}.

\begin{definition}
A monomial ideal $I$ is {\em stable} if for every $x^\mu\in I$ it satisfies that for each index $i>\min(x^\mu)$ we have that $x^{\mu}\frac{x_i}{x_{\min(x^\mu)}} \in I$, where $\min(x^\mu)$ denotes the index of the first variable \footnote{This is by definition the {\em class} of $x^\mu$. We use interchangeably $\min(x^\mu)$ and $\cls(x^\mu)$ to denote the same variable, since the first is uniformly used in the general literature on monomial ideals, and the latter is the usual denomination in the Pommaret bases jargon. We believe that this generates no confusion.} that divides $x^\mu$.
\end{definition}

\remark{Quasi-stable ideals can also be characterized in a combinatorial way, see \cite{HPV03,C09}. A monomial ideal is quasi-stable if and only if for any monomial $x^\mu\in I$ and for any $1\leq i < j \leq n$ such that $x_i$ divides $x^\mu$, there exists an integer $t>0$ such that $x_j^t x^\mu / x_i^{\mu_i}\in I$. }

\begin{proposition}[cf. \cite{S10}, Proposition 5.5.6]
A monomial ideal $I$ is stable if and only if its minimal monomial generating set is also a Pommaret basis for $I$.
\end{proposition}

\subsection{The Eliahou-Kervaire resolution for stable ideals}

The main object under study in this paper is the Pommaret-Seiler resolution, which is a non-minimal resolution for quasi-stable ideals that can be immediately read off from their Pommaret bases. It is a direct generalization of the Eliahou-Kervaire resolution \cite{EK90}, that gives an explicit closed form of the minimal free resolution of stable ideals. For the sake of completeness we include here the description of the Eliahou-Kervaire resolution, so that it is evident how the structure of the Pommaret-Seiler one is based on it. In both cases there is an important role played by a correspondence between elements of the generating set considered to build the resolution (the minimal generating set in the case of the Eliahou-Kervaire resolution, the Pommaret bases in the case of the Pommaret -Seiler resolution) and the set of monomials in the ideal. This correspondence gives rise to a partition of the set of monomials in the ideal, that plays an important role in the proof that these resolutions have a cellular structure, as we will see in Section \ref{sec:cellular}.

In order to describe the Eliahou-Kervaire resolution, we need the following result and definition (we follow the notations in \cite{M10}).
\begin{proposition}\label{prop:stable-partition}
Let $I$ be a stable monomial ideal and $x^\mu\in I$. Then there exists a unique element $g$ of the minimal monomial generating set of $I$, and a monomial $h$ such that $x^\mu=gh$ and for every $x_i$ dividing $h$ we have that $i\leq\min(g)$.
\end{proposition}

We say that $g$ is the {\em beginning} of $x^\mu$ and that $h$ is the {\em end} of $x^\mu$. We denote them by $\beg(x^\mu)$ and $\endfin(x^\mu)$ respectively.

\begin{definition}
Let $I$ be a monomial ideal. An $EK$-symbol for $I$ is a pair of the form $[f,u]$ where $f$ is a minimal generator of $I$ and $u$ is a square-free monomial satisfying $\min(u)\gneq \min(f)$.
\end{definition}

The Eliahou-Kervaire resolution is of the form 
\[
0\longrightarrow \cdots \longrightarrow E_l\longrightarrow E_{l-1}\longrightarrow\cdots E_0\longrightarrow I\longrightarrow 0,
\]
where each of the modules $E_j$ is a free module generated by the set of $EK$-symbols $[f,u]$ such that $\deg(u)=j$. The differential of the resolution is given by
\begin{equation}
d([f,u])=\sum_{x_i\vert u} \sgn(x_i,u)x_i[f,\frac{u}{x_i}]-\sum_{x_i\vert u}\sgn(x_i,u)\endfin(x_if)[\beg(x_if),\frac{u}{x_i}],
\end{equation}
where $\sgn(x_i,u)=1$ if the cardinality of the set $\{x_j \text{ s.t.  } x_j \mbox{ divides } u \mbox{ and } j\geq i\}$ is odd, and $-1$ otherwise \cite{EK90}. This complex is a minimal free resolution for stable ideals.

An important fact about the Eliahou-Kervaire resolution is that it can be obtained as an iterated mapping cone \cite{CE95,HT02}, in fact this is a way to prove its minimality, as seen in \cite{PS08}.

\subsection{The Pommaret-Seiler resolution}
For any quasi-stable ideal $I$ with Pommaret basis $\Hk=\{h_1,\dots,h_r\}$, Seiler gave in \cite[Section 7]{S09b} an explicit description of a graded free resolution of $I$ that can be directly read off from $\Hk$. In order to describe it we need the following notations.

Due to the properties of involutive divisions, we have that for any monomial $h_\alpha$ in $\Hk$ and any non-multiplicative variable $x_k\in\overline{\Xk}_P(h_\alpha)$ there exists a unique index $\Delta(\alpha,k)$ and a unique monomial $t_{\alpha;k}\in \kb[\Xk_P(h_{\Delta(\alpha,k)})]$ such that $x_k h_\alpha=t_{\alpha;k}h_{\Delta(\alpha,k)}$. Observe that this is closely related to the property described in Proposition \ref{prop:stable-partition} for stable ideals. The monomial $h_{\Delta(\alpha,k)}$ extends to quasi-stable ideals the role of $\beg(x_kh_\alpha)$ and $t_{\alpha;k}$ the role of $\endfin(x_kh_\alpha)$.

Let now denote by $\beta_0^{(k)}$ the number of generators in $\Hk$ of class $k$,  and by $d=\min\{k\vert\beta_0^{(k)}>0\}$ the minimal class of any generator in $\mathcal{H}$. Then the Pommaret-Seiler resolution has the form
\[
0\longrightarrow R^{r_{n-d}}\longrightarrow \cdots \longrightarrow R^{r_1}\longrightarrow\cdots R^{r_0}\longrightarrow I\longrightarrow 0,
\]
where the ranks of the free modules in the resolution are given by 
\[
r_i=\sum_{k=1}^{n-i}{{n-k}\choose{i}}\beta_0^{(k)}.
\]

The generators of the $i$-th free module in the Pommaret-Seiler resolution are given by pairs of the form $[h_\alpha,u]$ where $h_\alpha\in\Hk$ and $u$ is a degree $i$ square-free monomial satisfying $\min(u)\gneq \cls(h_\alpha)$. The differential in this resolution is given by the following formula (cf. \cite{S09b}):

\begin{equation}\label{eq:PSdifferential}
\partial([h_\alpha,u])=\sum_{j=1}^i(-1)^{i-j} \left( {u_j}[h_\alpha,\frac{u}{{u_j}}]-t_{\alpha;u_j}[h_{\Delta(\alpha,u_j)},\frac{u}{{u_j}}]\right),
\end{equation}

where by abuse of notation, $u_j$ denotes $t_{\alpha;u_j}$ and $\Delta(\alpha,u_j)$ the index of the corresponding variable (i.e. the $j$-th variable dividing $u$).

This resolution will be denoted by $(\mathbb{P},\partial)$. It is not minimal in general. In fact, it is minimal if and only if $I$ is a stable ideal, in which case it coincides with the Eliahou-Kervaire resolution (see \cite{S09b}, Theorem 8.6). Despite its non-minimality, one can read off most of the fundamental homological invariants from the Pommaret-Seiler resolution, like the projective dimension and Castelnuovo-Mumford regularity of $I$.

\begin{remark}
Observe that the presentation of the differentials in $(\mathbb{P},\partial)$ is slightly different from that in \cite[Theorem 7.2]{S09b}, since we use the degree $i$ squarefree monomials $u$ to stress the relation with E-K symbols, while \cite{S09b} uses a presentation of the generators based on the exterior algebra over a free $R$-module of rank $n$, giving in fact a bi-grading for $(\mathbb{P},\partial)$. Both presentations are equivalent, each of them highlighting a different aspect of the Pommaret-Seiler resolution.
\end{remark}

\begin{example}
Consider the ideal $I=\langle x_1^2, x_2^3\rangle$ from Example \ref{ex:quasi-stable-ideal}. It is not a stable ideal, since for instance $x_1^2\in I$ but  $x_1x_2\notin I$. The complex generated by the Eliahou-Kervaire symbols is therefore not a resolution of $I$. Consider instead the Pommaret-Seiler resolution of $I$. The Pommaret basis of this ideal is given by $\Hk=\{x_1^2,x_1^2x_2,x_1^2x_2^2,x_2^3\}$. The generators of the Pommaret-Seiler resolution of $I$ are indexed by $[x_1^2],[x_1^2x_2],[x_1^2x_2^2],[x_2^3]$ for the $0$-th module and by $[x_1^2,x_2],[x_1^2x_2,x_2],$ and $[x_1^2x_2^2,x_2]$, for the $1$-st module. Using formula (\ref{eq:PSdifferential}) for the differential, we obtain the following resolution, sorting the generators lexicografically:
\[
0\longrightarrow R^3 \stackrel{\partial_1}{\longrightarrow}R^4\stackrel{\partial_0}{\longrightarrow}I\longrightarrow 0,
\]
where

\[
\partial_0=
\begin{pmatrix}
    x_1^2  &  x_1^2x_2 & x_1^2x_2^2 & x_2^3   
\end{pmatrix}
,\;\;
 \partial_1=
\begin{pmatrix}
    -x_2  & 0 & 0       \\
    1  &  -x_2 &  0       \\
    0  &  1 & -x_2       \\
    0  &  0 & x_1^2       \\
\end{pmatrix} 
.
\]
It is, as expected, a non-minimal resolution for $I$.
\end{example}

\section{The Pommaret-Seiler resolution is cellular}\label{sec:cellular}

As with the Eliahou-Kervaire resolution, the Pommaret-Seiler resolution of a quasi-stable ideal $I$ can also be obtained as an iterated mapping cone for an adequate sorting of the generators of the Pommaret basis $\Hk$ of $I$ \cite{AFSS15}. This sorting is known as a $P$-ordering, and sorts the elements of $\Hk$ first by class and within each class lexicographically. Therefore, this resolution inherits the mapping cone property from the Eliahou-Kervaire resolution, which gives a stronger meaning to the statement that the Pommaret-Seiler resolution generalizes the Eliahou-Kervaire's one. In this section we prove that this generalization also applies to the cellular character of the resolution.

Cellular resolutions are based on the construction of geometric objects that encode the combinatorics of the generators of monomial ideals and of their least common multiples. The concept originated in \cite{BPS98} and was extended in \cite{BS98,MS04,JW09}. The geometric objects that support cellular resolutions can be simplicial complexes, polyhedral complexes or in general, CW-complexes. In this section we follow the definitions in \cite{BW02}.

Let $X$ be a CW-complex, such that $X^{(i)}$ is its set of $i$-cells and $X^*:=\bigcup_{i\leq0}X^{(i)}$ is the set of cells of $X$. We consider $X^*$ ordered by $c'\leq c$ iff $c'$ is contained in the closure of $c$, for $c,c'\in X^*$. We say that $X$ is $\NN^n$ graded if there is a mapping, denoted by $\gr$, from $X$ to $(\NN^n,\leq)$ compatible with the described ordering in $X$.
\begin{definition}
An $\NN^n$-graded resolution $\Pk$ of the $R$-module $M$ is {\em cellular} if there is an $\NN^n$-graded CW-complex $(X,\gr)$ such that
\begin{enumerate}[label=\alph*)]
\item There is a basis $e_c$ of $\Pk_i^\alpha$ indexed by the $i$-cells $c$ in $X^{(i)}$ such that $gr(c)=\alpha$.
\item {For $c\in X^{(i)}$ we have
\[
\partial_ie_c=\sum_{c\geq c'\in X^{(i-1)}}[c:c']x^{gr(c)-gr(c')} e_{c'},\\;\; i\geq 1,
\]}
\end{enumerate}

\end{definition}
where $[c:c']$ is the coefficient of $c'$ in the differential $\delta(c)$ of the cellular homology of $X$. For $c\in X^{(0)}$ we have that $\partial_oe_c$ has multidegree $\gr(c)$.


A prominent example of CW-complexes used for cellular resolutions of monomial ideals are polyhedral cell complexes, in particular geometrical realizations in $\RR^m$ of $m$-dimensional simplicial complexes. As with simplicial complexes, any polyhedral cell complex $X$ comes equipped with a reduced chain complex, in which the signs for the boundary differentials are specified by imposing an orientation on the faces of $X$, i.e. given a face $F$ in $X$ the boundary differential is the signed sum of its facets:
\[
\partial(F)=\sum_{facets\, G\subset F} sign(G,F)G,
\]
where the signs are induced by the chosen orientation in $X$.
We now label each vertex $v_1,\dots, v_r$ of $X$ with a monomial $x^{\mu_1},\dots, x^{\mu_r}$ in $R=\kb[x_1,\dots,x_n]$. In such a labeled complex, the label of an arbitrary face $F$ of $X$ is the monomial $x^{\mu_F}$ given by the least common
multiple $\lcm(x^{\mu_i} \vert v_i\in F)$ of the monomial labels of the vertices in $F$. 


\begin{example}
Let $I=\langle x^{\mu_1},\dots,x^{\mu_r}\rangle$ be a monomial ideal. Let $X$ be the complete $(r-1)$-dimensional simplex. The cellular resolution supported on $X$ and labeled by the generators of $I$ is a (possibly non-minimal) resolution of $I$. This is the cellular expression of the well known Taylor resolution \cite{T66}.

For the ideal $I=\langle x_1^2,x_2^3\rangle$ of Example \ref{ex:quasi-stable-ideal}, the simplex in two vertices supports the Taylor resolution of $I$ which is minimal in this case. The resolution has the form

\[
0\longrightarrow R\stackrel{\partial_1}{\longrightarrow}R^2\stackrel{\partial_0}{\longrightarrow}I\longrightarrow 0,
\]

where

\[
\partial_0=
\begin{pmatrix}
    x_1^2  &   x_2^3   
\end{pmatrix}
,\;\;
 \partial_1=
\begin{pmatrix}
    -x_2^3 \\
    x_1^2   \\
\end{pmatrix} 
.
\]
\end{example}

Mermin proves in \cite{M10} that the Eliahou-Kervaire resolution is cellular and gives an explicit cellular structure for it. Another cellular structure for the Eliahou-Kervaire resolution can be found in \cite{BW02}. In their paper \cite{DM14}, Dochtermann and Mohammadi give a sufficient condition (possession of a regular decomposition function) for an iterated mapping cone resolution to be cellular and prove that the Eliahou-Kervaire resolution satisfies this condition. We slightly generalize here the argument in \cite{DM14} by defining a regular decomposition function for the Pommaret-Seiler resolution and thus proving the main result of this section.

We first recall the mapping cone construction: Let $I$ be a monomial ideal and $G(I)=\{ m_1,\dots,m_r\}$ a generating set. Then for $j=1,\dots,r$ there are exact sequences:
\[0 \rightarrow R/\left( I_{j-1} :\langle m_j \rangle \right) \xrightarrow[]{\psi} R / I_{j-1} \rightarrow R/I_j \rightarrow 0,\]
where $I_j=\langle m_1,\dots,m_j\rangle$ for each $j=1,\dots,r$..
Thus, assuming that free resolutions $\Fk$ of $R/I_{j-1}$ and $\Gk$ of $ R/\left( I_{j-1} :\langle m_j \rangle \right)$ are already known, one obtains the free resolution of $R/I_j$ as the mapping cone of the chain complex homomorphism $\Gk\rightarrow \Fk$ lifting the map $\psi$. This construction yields an iterative procedure to compute a resolution of $R/I$ provided that for each $j$, we already know a resolution of $ R/\left( I_{j-1} :\langle h_j \rangle \right)$ as well as the liftings of the corresponding maps $\psi$ (which are the hard to obtain pieces in this procedure, see \cite{HT02}).

We say that a monomial ideal $I$ has {\em linear quotients} with respect to a (non-necessarily minimal) generating set $G(I)$ if there exists an ordering of $G(I)=\{ m_1,\dots,m_r\}$, such that for each $j\leq r$ the colon ideal $I_{j-1}:\langle m_j\rangle$ is generated by a subset of the variables. In such case, we denote by $\set(m_j)$ the set of variables that generate $I_{j-1}:\langle m_j\rangle$. The iterated mapping cone of a monomial ideal that has linear quotients with respect to its minimal generating set is a minimal free resolution of $R/I$, cf. \cite{HT02}, Lemma 1.5 and \cite{DM14}, Lemma 2.3. 

For any monomial ideal, let $M(I)$ be the set of monomials in $I$ and $G(I)$ a set of monomial generators of $I$. A {\em decomposition function} for an ideal $I$ with linear quotients with respect to $G(I)$ is an assignment $b: M(I)\rightarrow G(I)$ that for each monomial of $M(I)$ selects a unique generator in $G(I)$. We say that the decomposition function $b$ is {\em regular} if for each $m\in G(I)$ and every $x_t\in\set(m)$ we have that $\set(b(x_tm))\subseteq \set(m)$. Theorem 2.7 in \cite{DM14} (see also Theorem 1.12 in \cite{HT02}) gives a closed form formula for the differentials in the minimal free resolution (obtained as an iterated mapping cone) of any monomial ideal that has linear quotients with respect to an ordering of its minimal generating set and for which we can define a regular decomposition function. Furthermore, for any such ideal, both \cite{DM14} and \cite{G15} independently showed that the iterated mapping cone resolution is cellular by explicitly constructing a regular cell complex supporting it:

\begin{theorem}[\cite{DM14}, Theorem 3.11]
Suppose $I$ has linear quotients with respect to some ordering $(m_1,\dots,m_r)$ of the minimal generators, and furthermore suppose that $I$ has a regular decomposition function. Then the minimal resolution of $I$ obtained as an iterated mapping cone is cellular and supported on a regular $CW$-complex.
\end{theorem}

Any quasi-stable monomial ideal $I$ has linear quotients with respect to its $P$-ordered Pommaret basis, and the colon ideals are generated by the non-multiplicative variables of the corresponding generator:

\begin{proposition}[\cite{AFSS15}, Proposition 7.2 and \cite{HSS12}, Proposition 26]
Let $\Hk=\{h_1,\dots,h_r\}$ be a $P$-ordered monomial Pommaret basis of the quasi-stable monomial ideal $I$. Then $I$ possesses linear quotients with respect to the basis $\Hk$ and $\langle h_{\alpha+1},\dots,h_r\rangle:\langle h_\alpha\rangle=\langle\overline{\Xk}_p(h_\alpha)\rangle$ for all $\alpha=1,\dots,r-1$.
\end{proposition}

Observe that in this case, $\Hk$ is in general a non-minimal generating set of $I$, and hence the iterated mapping cone resolution obtained (i.e. the Pommaret-Seiler resolution of $I$, see \cite{AFSS15}) is not the minimal free resolution. Minimality is only obtained if $I$ is stable. Using the iterated mapping cone structure of the Pommaret-Seiler resolution we can prove that it is cellular by defining a regular decomposition function for $I$ with respect to its Pommaret basis $\Hk$.

\begin{theorem}\label{th:cellular}
The Pommaret-Seiler resolution of a quasi-stable monomial ideal is cellular.
\end{theorem}
\begin{proof}

Let $I$ be a quasi-stable ideal, $\Hk=\{h_1,\dots,h_r\}$ its Pommaret basis and $M(I)$ the set of monomials in $I$. The fact that $\Hk$ is an involutive basis for $I$ means that $M(I)=\bigcup_{h_i\in\Hk} \Ck(h_i)$, where the union is disjoint, hence every monomial in $M(I)$ has a unique involutive divisor in $\Hk$.

We now define the decomposition function $b:M(I)\rightarrow \Hk$ as $b(x^\mu)=h_\alpha$ where $h_\alpha$ is the unique element of $\Hk$ that is an involutive divisor of $x^\mu$. To see that $b$ is regular observe that $\set(h_\alpha)=\overline{\Xk}_P(h_\alpha)$. Now, for each $x_t\in\overline{\Xk}_P(h_\alpha)$ we have that $b(x_th_\alpha)=h_{\Delta(\alpha,t)}$ and $\cls(h_{\Delta(\alpha,t)})\geq \cls(h_\alpha)$, hence $\set(h_{\Delta(\alpha,t)})\subseteq \set(h_\alpha)$ and $b$ is regular. 

For every monomial ideal with linear quotients and a regular decomposition function $b$ there is an explicit expression of the differential in the mapping cone resolution, that is given in Theorem 1.12 \cite{HT02} which, in our case has the following expression:
\begin{equation}\label{eq:HTdifferential}
\partial([h,u])=\sum_{j=1}^i(-1)^{i-j} \left( {u_j}[h,\frac{u}{{u_j}}]-\frac{{u_j}\cdot h}{b({u_j}\cdot h)}[b({u_j}\cdot h),\frac{u}{{u_j}}]\right),
\end{equation}
for $h\in\Hk$ and $u$ is any squarefree monomial formed  only by non-multiplicative variables of $h$. Observe that $b({u_j}\cdot h)=h_{\Delta_{h,u_j}}$ and $\frac{{u_j}\cdot h}{b({u_j}\cdot h)}=t_{h,u_j}$, and hence this explicit description of the differentials is identical to the one for the Pommaret-Seiler resolution shown in (\ref{eq:PSdifferential}).

The rest of the proof follows the lines of the proof of Theorem 3.11 in \cite{DM14} except for the minimality of the resolution. It is a constructive proof, in the sense that it consists on building a regular CW-complex that supports the resolution using the iterated mapping cone construction. The procedure gives a geometric interpretation of the construction of the iterated algebraic mapping cone for ideals with a regular decomposition function as seen in \cite{DM14} and \cite{G15}. 

Assuming that the ideal $I$ has linear quotients with respect to the generating set $\Hk=\{h_1,...,h_r\}$, then since $ \left( I_{j-1} :\langle h_j \rangle \right)$ is generated by a set of the variables, its minimal free resolution can be easily obtained by computing the Taylor resolution. This minimal resolution is in fact supported on a simplex with as many vertices as elements in $set(h_j)$, hence it is simplicial. The lifted map between the resolutions of $ R/\langle set(h_j) \rangle$ and $R/I_{j-1}$ induces a cellular map of $\Delta$ into $X_{j-1}$, where $\Delta$ represents the simplex associated to the Taylor resolution of $ R/\langle set(h_j) \rangle$, and $X_{j-1}$ is the CW-complex inductively constructed in the previous steps of the iterated mapping cone construction. Therefore, the mapping cone construction can be realized as an iterative procedure where in each step a geometric simplex is glued onto the cellular resolution already built in the previous steps. The fact that this resolution is minimal depends on whether $\Hk$ is the minimal generating set of $I$ (i.e. whether $I$ is stable) by Lemma 1.5 in \cite{HT02}.\\
The regular decomposition function $b$ defined for $I$ knowing that it has linear quotients is now used to explicitly describe the geometric differential (boundary operator) of the CW-complex constructed by gluing iteratively the Taylor simplices obtained by the above procedure. This description is given in Theorem 1.12 \cite{HT02} (see also Theorem  2.6 \cite{DM14} or Construction 2.5 \cite{G15}).

In the case that $I$ is a quasi-stable ideal and $\Hk$ its Pommaret basis, we extend this construction to the Pommaret-Seiler resolution by means of the so-called $P$-graph of the Pommaret basis, which is defined in \cite[Section 5]{S09b}, see also \cite{PR05,C21}. The $P$-graph of $\Hk$, consists of a vertex for each $h_\alpha\in\Hk$ and a directed edge from $h_\alpha$ to $h_{\Delta(\alpha,t)}$ for each $h_\alpha$ in $\Hk$ and each $x_t\in\overline{\Xk}_P(h_\alpha)$. The $0$-cells of the $CW$-structure for the Pommaret-Seiler resolution of $I$ are the vertices of the $P$-graph of $\Hk$ considered as points in $\mathbb{R}^n$.
To form the higher dimensional cells, we follow the construction in \cite{DM14} and \cite{M10}. Let $h\in\Hk$, $\tau=\{j_1,\dots,j_p\}$ such that $\{x_{j_1},\dots,x_{j_p}\}\subseteq\overline{\Xk}_P(h)$ with $j_1<\cdots<j_p$ and $\sigma$ a permutation of $\tau$. We define $\ch(h,\tau,\sigma)$ to be the subset of $\mathbb{R}^n$ obtained as the convex hull of the elements of $\Hk$ that we reach by applying the decomposition function $b(x_th_i)=h_{\Delta(h_i,t)}$ in the order prescribed by $\sigma$, i.e. applying first $b(x_{j_1}h)$, then $b(x_{j_2}h_{\Delta(h,j_1)})$ and so on. If there are no repetitions of elements of $\Hk$ involved in the description of $\ch(h,\tau,\sigma)$ then, because of the regularity of the decomposition function of $I$, $\ch(h,\tau,\sigma)$ is a $p$-dimensional simplex, and we say that $\ch(h,\tau,\sigma)$ is non-degenerate. Otherwise, we say that $\ch(h,\tau,\sigma)$ is degenerate and it is in fact a face of $\ch(h,\tau,\sigma')$ where $\sigma'$ is another permutation of $\tau$ such that $\ch(h,\tau,\sigma')$ is non-degenerate.

Now, the cell $U(h,\tau)$ is defined as the union of the $\ch(h,\tau,\sigma)$ over all permutations $\sigma$ of $\tau$. For these cells we have a geometrical differential map
\[
d(U(h,\tau))=\sum_{i=1}^p(-1)^iU(h,\tau\setminus j_i)-\sum_{i=1}^p(-1)^iU(h_{\Delta(h,j_i)},\tau\setminus j_i).
\]
Finally, by adding the monomials that correspond to the labelling of each cell in the CW-complex in these differentials, we obtain the differential in (\ref{eq:PSdifferential}) and have that the described structure is indeed the CW-structure that supports the Pommaret-Seiler resolution of $I$, which in our case has the following expression:
\begin{equation}\label{eq:CWdifferential}
\partial([h,u])=\sum_{j=1}^i(-1)^{i-j} \left( {u_j}[h,\frac{u}{{u_j}}]-\frac{{u_j}\cdot h}{b({u_j}\cdot h)}[b({u_j}\cdot h),\frac{u}{{u_j}}]\right),
\end{equation}
for $h\in\Hk$ and $u$ any squarefee monomial formed by non-multiplicative variables of $h$. One can easily check that this explicit description of the differentials is identical to the one for the Pommaret-Seiler resolution, given by \cite{S09b}.
\end{proof}

\section{Reduction of the Pommaret-Seiler resolution}\label{sec:reduction}
The Pommaret-Seiler resolution for quasi-stable ideals is known to be non-minimal, nevertheless, some of the homological invariants of this resolution can be read off directly from it. In particular, we know that it is a resolution of minimal length \cite{S09b}. If we denote by $\Cf_I$ the CW-complex described in the previous section, then we have that the dimension of $\Cf_I$ equals the projective dimension of $I$. In this section we show that $\Cf_I$ can be reduced using Discrete Morse Theory \cite{F01}, and since this complex supports the Pommaret-Seiler resolution, we can equivalently use the algebraic formulation in \cite{JW09}. This fact suggests an algorithm for reducing the Pommaret-Seiler resolution so that we obtain the minimal one. This kind of reductions based on Discrete Morse Theory has been used before in the context of monomial ideals, see \cite{AFG20,BM20,CEFMMSS22}.

Let $I$ be a quasi stable ideal, $\mathcal{H}$ its finite Pommaret basis and $d$ the minimal class of any generator in $\mathcal{H}$. Let $\Gamma_{\mathbb{P}}$ be the directed graph associated to $(\mathbb{P},\partial)$ the Pommaret-Seiler resolution of $I$, as described in \cite[Section 2]{S06}, i.e. the vertices of $\Gamma_{\mathbb{P}}$ are the generators of ${\mathbb{P}}$ and there is an edge from $[h_\alpha,u]$ to $[h_\beta,v]$ if the coefficient of $[h_\beta,v]$ in $\partial([h_\alpha,u])$ is different from $0$. Observe that in such case $\deg(u)=\deg(v)+1$. Following \cite{S06} we will use the following terminology.

\begin{definition}
A {\em partial matching} on a directed graph $G=(V, E)$ is a subset $A\subseteq E$ of the edges of $G$ such that any vertex is incident to at most one edge in $A$. For such a partial matching we denote by $G^A=(V, E^A)$ the directed graph with the same vertices and edges as $G$ and such that the edges in $A$ are reversed.

We define $A^+\subseteq V$ to be the subset of those vertices that are the targets of the reversed arrows in $A$, and $A^-\subseteq V$ to be the set of their sources. Finally $A^0\subseteq V$ is the set of vertices of $G$ which are not incident to any of the reversed arrows. These are called {\em critical} vertices.

A Morse matching on the directed graph $\Gamma_{\mathbb{P}}$ is a partial matching $A$ satisfying the following conditions:
\begin{itemize}
\item For every edge $a\rightarrow b$ where $a\in \mathbb{P}_i$, $b\in \mathbb{P}_{i-1}$, the projection of the map $\partial(a)$ to the component of $\mathbb{P}_{i-1}$ generated by $b$ is an isomorphism.
\item Every set of generators $\mathbb{P}_i$ possesses a well-founded partial order $<$ such that for any $a,c\in \mathbb{P}_i$ for which there is a path $a\rightarrow b\rightarrow c$ in $\Gamma_{\mathbb{P}}^A$ we have that $c<a$. We say that such an order respects the matching $A$.
\end{itemize}
\end{definition}

The idea behind this construction is that any Morse matching induces the construction of a smaller complex with the same homology of $(\mathbb{P},\partial)$, which is called a Morse reduction of  $(\mathbb{P},\partial)$ by $A$. If the resulting complex has no isomorphisms then it cannot be further reduced, the new complex is a minimal resolution of $I$ and we say that $A$ is a {\em perfect Morse matching}.

With these definitions at hand, we describe a procedure to reduce $(\mathbb{P},\partial)$ to a minimal resolution of the corresponding ideal $I$. We proceed variable by variable, starting from $x_n$ down to $x_{d+1}$. Consider the following partial matching in $\Gamma_{\mathbb{P}}$:

\[
V_n=\{[h_\alpha,u]\rightarrow [h_\beta,u/x_n] \mbox{ if the coefficient of } [h_\beta,u/x_n] \mbox{ in } \partial([h_\alpha,u]) \mbox{ is } 1 \mbox{ or }-1\}.
\]

Now for each variable $x_i$ with $i$ from $n-1$ down to $d+1$ consider 
\begin{align*}
 V_i=&\{[h_\alpha,u]\rightarrow [h_\beta,u/x_i] \mbox{ if the coefficient of } [h_\beta,u/x_i] \mbox{ in } \partial([h_\alpha,u]) \mbox{ is } 1 \mbox{ or }-1\\
&\mbox{ and none of } [h_\alpha,u],\, [h_\beta,u/x_i] \mbox{ incides in an edge in } V_j \mbox{ for } j>i\}.   
\end{align*}

It is clear that $V=\bigcup_{i=d+1}^{n}V_i$ is a partial matching on $\Gamma_{\mathbb{P}}$, the fact that it is a (perfect) Morse matching needs to be proven.

\begin{proposition}\label{prop:matching}
The partial matching $V=\bigcup_{i=d+1}^n V_i$ is a perfect Morse matching in $\Gamma_{\mathbb{P}}$.
\end{proposition}
\begin{proof}
In order to show that it is a Morse matching, observe that from Lemma 1 in \cite{S06}, it is enough to show that there are no directed cycles in $\Gamma_{\mathbb{P}}^{V}$, the digraph obtained from $\Gamma_{\mathbb{P}}$ by reversing all the edges in $V$.

For this, first observe that if any such cycle exists, then all the edges involved occur between some pair of consecutive free modules $P_a$ and $P_{a+1}$ for some $a$ in $\{1,\dots,\pd(I)\}$. This is because whenever there is an edge from any $[h,u]$ in $P_{a-1}$ to $[h',ux_i]$ in $P_a$, there cannot be another edge from $[h',ux_i]$ to $[h'',ux_ix_j]$, due to the definition of $V_i$ and $V_j$, since $[h',ux_i]$ is adjacent to a reversed edge in either $V_i$ or $V_j$ but not both. Any cycle in $\Gamma_{\mathbb{P}}^{V}$ involving $P_{a-1}$, $P_a$ and $P_{a+1}$ must contain at least two such consecutive edges, hence no such cycle exists.

Consider then cycles involving edges in $P_a$ and $P_{a+1}$ for some $a$. These have the form 
\[
[h,u]\overset{x_{i_1}}{\longmapsto}[h',u']\overset{x_{j_1}}{\longrightarrow}[h'',u'']\overset{x_{i_2}}{\longmapsto}\cdots\overset{x_{j_k}}{\longrightarrow}[h,u],
\]
where each ticked arrow $\mapsto$ indicates a (reversed) element of $V$ and the usual arrows $\rightarrow$ indicate regular edges in $\Gamma_{\mathbb{P}}$. The variable above each arrow indicates the variable that differs between $u$ and $u'$ etc. We call the ticked arrows {\em upward} arrows and the regular ones {\em downward} arrows. For any generator $[h,u]$ of $\mathbb{P}$, we call $h$ its symmetric part and $u$ is its exterior part.  We say that upward arrows {\em insert} variables in the exterior part $u$ of the corresponding generator $[h,u]$, and that downward arrows {\em extract} variables from it. Upward arrows always denote isomorphisms by definition, i.e. the coefficient of the source in the differential of the target is a nonzero scalar. Observe that if any downward arrow is not an isomorphism then the total degree of the target generator is strictly smaller than that of the source generator, while it remains constant for isomorphisms. Hence, in order to have a cycle and be back to $[h,u]$ we cannot have any non-isomorphism for degree reasons.

Therefore the cycle
\[
[h,u]\overset{x_{i_1}}{\longmapsto}[h',u']\overset{x_{j_1}}{\longrightarrow}[h'',u'']\overset{x_{i_2}}{\longmapsto}\cdots\overset{x_{j_l}}{\longrightarrow}[h,u],
\]
must be formed only by isomorphisms. Upward arrows are all elements of $V$ and downward arrows are not. Now, everytime we move by a ticked arrow, we introduce a variable, $x_{i_1}$, $x_{i_2}$ etc. into the exterior part (that started with $u$) of the corresponding generator. And every time we move by an usual arrow, we extract a variable $x_{j_1}$, $x_{j_2}$ etc. from the exterior part of the corresponding generator. Since to close the cycle we must get back to $u$ this means that we extracted exactly the same variables that we introduced in the exterior part (possibly in a different order). Observe that for any cycle and every consecutive pair of arrows $[h,u]\overset{x_{i_k}}{\longmapsto}[h',u']\overset{x_{j_k}}{\longrightarrow}[h'',u'']$ we have $x_{i_k}>x_{j_k}$. The reason is that if $x_{j_k}$ was bigger, then it would be an isomorphism involving $[h'',u'x_{j_k}]$ belonging to $V_{j_k}$ (or some other previous set in the construction process of $V$) which would have been reversed in the construction process of $V$ and therefore the upward arrow introducing $x_{i_k}$ would not be reversed, hence being a regular downward arrow. Let now $u_m$ and $u_M$ be the smallest and biggest variables in $u$ respectively. Without loss of generality, we can consider the cycle starting at $[h,u]\in P_a$. Observe that in such cycle, whenever we extract a variable $x_j$ via a downward arrow, the immediate previous arrow is an upward one that introduces a variable $x_i$ that must be bigger than $x_j$.  The first arrow of the cycle introduces $x_{i_1}$, which will have to be extracted later in the cycle. Note that $u_m<x_{i_1}<u_M$. The first inequality comes from the fact that the variable that we extract after introducing $x_{i_1}$ has to be smaller than it, as said before. For the second inequality, observe that if $x_{i_1}$ was bigger than any variable in $u$ we could not use any variable in $u$ to extract it later in the cycle, and therefore the cycle could not be completed. Now, when $x_{i_1}$ is extracted from the cycle, the previous upward arrow can be either a bigger variable that was introduced in $u$ or a variable already in $u$ which is bigger than $x_{i_1}$. In the first case, we have again a variable $x_i$ not in $u$ that we need to eventually extract, which satisfies that $x_i>x_{i_1}$, hence we obtain an ascending sequence of variables that we have to extract and this makes impossible to close the cycle. In the second case, we need to introduce a variable $x_j$ bigger than $x_{i_1}$ that was already in $u$. But this means that at some point we must have extracted $x_j$ using another variable $x_k$ either in $u$ or not in $u$. In both cases we end up again with an ascending sequence of variables that eventually reaches $u_M$ or some other bigger variable, and therefore the cycle can not be closed. In short, assuming there exists a cycle, the biggest variable in it must be extracted from the exterior part at some point, but for this we need to introduce a bigger variable just before that extraction, and this contradicts the fact that the extracted variable is the biggest one.

Finally, to see that $V$ is perfect observe that no isomorphisms remain in $\Gamma^V_{\mathbb{P}}$. In fact, any isomorphism $[h,u]\mapsto[h/x_i,ux_i]$ is either reversed at the step $V_i$ or at least one of $[h,u]$ or $[h/x_i,ux_i]$ is involved in step $V_j$, $j>i$ and therefore the isomorphism is not present in $\Gamma^V_{\mathbb{P}}$. 
\end{proof}

As a result of this proposition we have the main result in this section.

\begin{theorem}
$\Gamma^V_{\mathbb{P}}$ is a minimal free resolution of $I$.
\end{theorem}

\begin{proof}
Any Morse matching produces a reduction $\Gamma'_{\mathbb{P}}$ on $\Gamma_{\mathbb{P}}$, which is homotopically equivalent to it, i.e. both are resolutions of $I$. For each edge $[h,u]\rightarrow [h',u']$ in $V$, we say that $[h,u]$ is a source cell and $[h',u']$ is a target cell. All other generators (i.e. those that are not incident to any edge in $V$) are called {\em critical} cells. $\Gamma^V_{\mathbb{P}}$ is obtained from $\Gamma_{\mathbb{P}}$ by deleting the generators involved in $V$ and all the edges of $\Gamma_{\mathbb{P}}$ adjacent to them, i.e. the only remaining generators are those corresponding to critical cells with respect to $V$. Since $V$ is perfect, there are no isomorphisms in $\Gamma^V_{\mathbb{P}}$ and hence it is a minimal free resolution of $I$, whose generators correspond to the critical cells with respect to $V$. The differential in  $\Gamma^V_{\mathbb{P}}$ is obtained from that of  $\Gamma_{\mathbb{P}}$ by a standard algorithmic method whose (technical but straightforward) details are given in \cite{S06}. 
\end{proof}

\begin{remark}
Observe that in \cite{AFSS15} the authors use Discrete Morse Theory in the context of Pommaret-Seiler resolutions for general polynomial ideals. The problem is, however, different from the one treated here. In \cite{AFSS15} the authors start with a complex  given by Sk\"oldberg in \cite{S11} that is in fact a free resolution for modules with initially linear syzygies. This construction is then reduced by means of Discrete Morse Theory and the Pommaret-Seiler resolution is obtained. Hence, while in \cite{AFSS15} the reduction is from Sk\"oldberg's complex to the Pommaret-Seiler resolution, in our case we go, in the monomial case, from Pommaret-Seiler resolution to a minimal free resolution. This last reduction to a minimal free resolution is performed (but not described) in \cite{AFSS15} using linear algebra in the way described in \cite[Chapter 6]{CLO05}.
\end{remark}

\begin{remark}
The information for constructing $\Gamma_{\mathbb{P}}$ and $V$ can be read off from an annotated version of the $P$-graph of $I$. We can build this {\em annotated} $P$-graph as we compute the Pommaret basis of $I$. Let us denote by $G_I$ the P-graph of $I$, and by $(\mathbb{P},\partial)$ its Pommaret-Seiler resolution. As we build the Pommaret basis $\Hk$ we can store the information of $G_I$ assigning two pieces of information to each edge $e_{i,j}$, namely $k(e_{i,j})$ will denote the variable $k\in\overline{\Xk}_P(h_i)$ used to reach $h_j$ from $h_i$ and $t(e_{i,j})$ denotes the term $t_{i;k}$ for that $k$. 
Now, for each directed path $p=(p_1,\dots,p_l)=(e_{i_1,j_1},\dots,e_{i_l,j_l})$ in $G_I$ such that $k(e_{i_a,j_a})<k(e_{i_b,j_b})$ for $a<b$, we say that the multidegree of the path is $\md(p)=\prod_{k=1}^l t(e_{i_k,j_k})$. We say that a path between nodes $i$ and $j$ is a {\em valid path} if $\md(p)=1$. Valid paths indicate the isomorphisms that we consider for the Morse matching $V$ described above.

The annotated $P$ graph $G_I$ allows us to construct alternative Morse matchings in $\Gamma_{\mathbb{P}}$. For any multidegree $\mu$ consider then the following set of vertices in $\Gamma_{\mathbb P}$ (equivalently in $\Cf_I$): $V_{\mu}=\{\alpha:\{u\}\mid \md(\alpha:\{u\})=\mu\}$ where $\alpha:\{u\}$ indicates the vertex in $\Gamma_{\mathbb{P}}$ corresponding to the generator indexed by $[h_\alpha,u]$ in the Pommaret-Seiler resolution $\mathbb{P}$. Consider now the following partial matching in $V_\mu$: $E_\mu=\{\alpha:\{u\}\rightarrow \beta:\{u/u_j\}\mid j=\max(u)\}$. Then we have that 
$\bigcup_{\mu\in\mathbb{N}^n} E_\mu$ is a Morse matching in $\Gamma_{\mathbb{P}}$ (not perfect in general). Proceeding iteratively by multidegree, we obtain then a reduction of $\mathbb{P}$. If this reduction is already the minimal free resolution of $I$ then we stop the algorithm. Otherwise, we can proceed by further use of Morse matchings using those pairs of generators $[h_\alpha,u]$, $[h_\beta,u']$ in the reduced resolution such that the coefficient of $[h_\beta,u']$ in the differential of $[h_\alpha,u]$ is a nonzero scalar. These Morse mathings are always possible and they strictily reduce the number of such nonzero scalars in the differential of the resolution, hence the algorithm terminates after a finite number of steps and provides the minimal free resolution of $I$. The cellular structure $\Cf_I$ of $\mathbb{P}$ allows us to read this reduction in terms of the geometrical differential of $\Cf_I$, and it can be used to obtain other geometrical Morse matchings that can eventually reduce the resolution completely.
\end{remark}

\begin{example}
To illustrate the above notions, let $I=\langle x^2,y^4,y^2z^2,z^3\rangle$ a quasi-stable ideal whose Pommaret basis is 
\[
\Hk=\{x^2, x^2y,x^2y^2,x^2y^3,x^2z,x^2yz,x^2y^2z, x^2y^3z,x^2z^2, x^2yz^2,y^4,y^4z,y^2z^2,z^3\}.
\]

The generators of the Pommaret-Seiler resolution, i.e. the vertices of the graph $\Gamma_\mathbb{P}$, are indexed by the following symbols:

Level $0$: $[x^2], [x^2y],[x^2y^2],[x^2y^3],[x^2z],[x^2yz],[x^2y^2z],$ $ [x^2y^3z],[x^2z^2], [x^2yz^2]$, $[y^4],[y^4z],[y^2z^2],[z^3]$. These correspond to the vertices in the $P$-graph of $I$, hence the vertices of $\Cf_I$.

Level $1$: $[x^2,y], [x^2y,y],[x^2y^2,y],[x^2y^3,y],[x^2z,y],[x^2yz,y],[x^2y^2z,y]$, $ [x^2y^3z,y]$, $[x^2z^2,y], [x^2yz^2,y]$, $[x^2,z], [x^2y,z],[x^2y^2,z],[x^2y^3,z],[x^2z,z],[x^2yz,z],[x^2y^2z,z],$ $ [x^2y^3z,z],[x^2z^2,z], [x^2yz^2,z][y^4,z],[y^4z,z],[y^2z^2,z]$, which correspond to edges in the $P$-graph of $I$ and to $1$-cells in $\Cf_I$.

Level $2$: $[x^2,yz], [x^2y,yz],[x^2y^2,yz],[x^2y^3,yz],[x^2z,yz],[x^2yz,yz],[x^2y^2z,yz],$ $ [x^2y^3z,yz],[x^2z^2,yz], [x^2yz^2,yz]$, which correspond to the $2$-cells in $\Cf_I$.

Now following the procedure described in Proposition \ref{prop:matching}, we have that $V_3$ has thirteen elements and $V_2$ has five:
 \begin{align*}
 V_3=&\{[x^2,yz]\rightarrow[x^2z,y],[x^2y,yz]\rightarrow[x^2yz,y],[x^2y^2,yz]\rightarrow[x^2y^2z,y],\\
  &[x^2y^3,yz]\rightarrow[x^2y^3z,y],[x^2z,yz]\rightarrow[x^2z^2,y],[x^2yz,yz]\rightarrow[x^2yz^2,y],\\
  & [x^2,z]\rightarrow[x^2z],[x^2y,z]\rightarrow[x^2yz],[x^2y^2,z]\rightarrow[x^2y^2z],\\
  &[x^2y^3,z]\rightarrow[x^2y^3z],[x^2z,z]\rightarrow[x^2z^2],[x^2yz,z]\rightarrow[x^2yz^2],\\
  &[y^4,z]\rightarrow[y^4z]\}\\
  V_2=&\{[x^2y^2z,yz]\rightarrow[x^2y^3z,z],[x^2z^2,yz]\rightarrow[x^2yz^2,z],[x^2,y]\rightarrow[x^2y],\\
  &[x^2y,y]\rightarrow[x^2y^2],[x^2y^2,y]\rightarrow[x^2y^3]\}
 \end{align*}
 
The critical cells in $\Gamma^V_{\mathbb{P}}$, which support the minimal free resolution of $I$, are then
 \begin{align*}
& \{ [x^2],[y^4],[y^2z^2],[z^3],[x^2y^3,y],[x^2z^2,z],[x^2y^2z,z],[y^4z,z],\\
 &[y^2z^2,z],[x^2y^3z,yz],[x^2yz^2,yz] \}.
 \end{align*}

Figure \ref{fig:P-graph} shows the $P$-graph of $I$ where each edge $e_{i,j}$ is labelled by $t(e_{i,j})$, and Figure \ref{fig:P-graph_critical} shows the critical cells of the Morse reduction induced by the Morse-matching based on $V$ described in Proposition \ref{prop:matching}. We can see that the reduced complex, depicted in Figure \ref{fig:P-graph_skeleton} is a cell complex that supports the minimal free resolution of $I$. Hence, the geometrical differential of this complex gives us the differentials in the minimal free resolution of $I$, which is given by

\[
0\longrightarrow R^2\stackrel{\partial_2}{\longrightarrow}R^5 \stackrel{\partial_1}{\longrightarrow}R^4\stackrel{\partial_0}{\longrightarrow}I\longrightarrow 0,
\]
where

\[
\partial_0=
\begin{pmatrix}
    x^2  &  y^4 & y^2z^2 & z^3   
\end{pmatrix},
 \partial_1=
\begin{pmatrix}
    -y^4  &  0 & -y^2z^2 & -z^3 & 0      \\
    -x^2  &  -z^2 &  0 & 0 & 0      \\
    0  &  y^2 & x^2 & 0 & -z      \\
    0  &  0 & 0 & x^2 & y^2      \\
\end{pmatrix},
\]
\[
 \partial_2=
\begin{pmatrix}
    z^2  &  0  \\
    x^2  &  0  \\
    -y^2  &  z  \\
    0  &  -y^2 \\
    0 &  x^2  \\
\end{pmatrix} .
\]

\tikzstyle{vertex}=[circle,draw=black,fill=black!25,minimum size=20pt,inner sep=0pt,font=\small]
\tikzstyle{selected vertex} = [vertex, fill=red!24]
\tikzstyle{edge} = [draw,thick,-]
\tikzstyle{weight} = [font=\small]
\tikzstyle{selected edge} = [draw,line width=5pt,-,red!50]
\tikzstyle{ignored edge} = [draw,line width=5pt,-,black!20]

\begin{figure}
\centering
\begin{minipage}[b][11cm][s]{.45\textwidth}
\centering
\vfill

\begin{tikzpicture}[scale=1, auto,swap]

\foreach \pos/\name in {{(0,0)/x^2}, {(1,1)/x^2y}, {(2,2)/x^2y^2},
                            {(3,3)/x^2y^3}, {(4,4)/y^4}, {(2,0)/x^2yz}, {(3,1)/x^2y^2z}, {(4,2)/x^2y^3z}, {(5,3)/y^4z},{(1,-1)/x^2z},
                            {(2,-2)/x^2z^2},{(3,-1)/x^2yz^2},{(3,-3)/z^3},{(5,0)/y^2z^2}}
        \node[vertex] (\name) at \pos {$\name$};
        
\foreach \pos/\name in {{(0,0)/x^2}, {(4,4)/y^4}, {(3,-3)/z^3},{(5,0)/y^2z^2}}
        \node[vertex] (\name) at \pos {$\name$};

 \foreach \source/ \dest /\weight in {x^2/x^2y/1,x^2y/x^2y^2/1,x^2y^2/x^2y^3/1,x^2y^3/y^4/x^2,
 						     x^2/x^2z/1,x^2z/x^2z^2/1,x^2z^2/z^3/x^2,
						     x^2z/x^2yz/1,x^2yz/x^2y^2z/1,x^2y^2z/x^2y^3z/1,x^2y^3z/y^4z/x^2,
						     x^2y/x^2yz/1,x^2y^2/x^2y^2z/1,x^2y^3/x^2y^3z/1,
						     x^2yz/x^2yz^2/1,
						     x^2y^2z/y^2z^2/x^2,x^2y^3z/y^2z^2/x^2y,y^4z/y^2z^2/y^2,
						     x^2z^2/x^2yz^2/1,x^2yz^2/z^3/x^2y,y^4/y^4z/1,
						     x^2yz^2/y^2z^2/x^2,y^2z^2/z^3/y^2
                                         }
 \path[edge] (\source) -- node[weight] {$\weight$} (\dest);

\foreach \source/ \dest /\weight in {x^2y^3/y^4/,y^4z/y^2z^2/,x^2yz^2/y^2z^2/,y^2z^2/z^3/}                                         
 \path[edge] (\source) -- node[weight] {$\weight$} (\dest);

\end{tikzpicture}
\vfill
\caption{$P$-graph of the ideal $I=\langle x^2,y^4,y^2z^2,z^3\rangle$.}\label{fig:P-graph}
\vspace{\baselineskip}
\end{minipage}\qquad
\begin{minipage}[b][11cm][s]{.45\textwidth}
\centering
\vfill
\begin{tikzpicture}[scale=1, auto,swap]

\draw[fill=red!25] (4,2)--(5,3)--(5,0);
\draw[fill=red!25] (5,0)--(3,-1)--(3,-3);

\foreach \pos/\name in {{(0,0)/x^2}, {(1,1)/x^2y}, {(2,2)/x^2y^2},
                            {(3,3)/x^2y^3}, {(4,4)/y^4}, {(2,0)/x^2yz}, {(3,1)/x^2y^2z}, {(4,2)/x^2y^3z}, {(5,3)/y^4z},{(1,-1)/x^2z},
                            {(2,-2)/x^2z^2},{(3,-1)/x^2yz^2},{(3,-3)/z^3},{(5,0)/y^2z^2}}
        \node[vertex] (\name) at \pos {$\name$};
        
\foreach \pos/\name in {{(0,0)/x^2}, {(4,4)/y^4}, {(3,-3)/z^3},{(5,0)/y^2z^2}}
        \node[selected vertex] (\name) at \pos {$\name$};

 \foreach \source/ \dest /\weight in {x^2/x^2y/,x^2y/x^2y^2/,x^2y^2/x^2y^3/,x^2y^3/y^4/,
 						     x^2/x^2z/,x^2z/x^2z^2/,x^2z^2/z^3/,
						     x^2z/x^2yz/,x^2yz/x^2y^2z/,x^2y^2z/x^2y^3z/,x^2y^3z/y^4z/,
						     x^2y/x^2yz/,x^2y^2/x^2y^2z/,x^2y^3/x^2y^3z/,
						     x^2yz/x^2yz^2/,
						     x^2y^2z/y^2z^2/,x^2y^3z/y^2z^2/,y^4z/y^2z^2/,
						     x^2z^2/x^2yz^2/,x^2yz^2/z^3/,y^4/y^4z/,
						     x^2yz^2/y^2z^2/,y^2z^2/z^3/
                                         }
 \path[edge] (\source) -- node[weight] {$\weight$} (\dest);

\foreach \source/ \dest /\weight in {x^2y^3/y^4/,y^4z/y^2z^2/,x^2y^2z/y^2z^2/,y^2z^2/z^3/,x^2z^2/z^3/}                                         
 \path[selected edge] (\source) -- node[weight] {$\weight$} (\dest);

\end{tikzpicture}
\vfill
\caption{$P$-graph of the ideal $I=\langle x^2,y^4,y^2z^2,z^3\rangle$ with critical faces of Morse reduction.}\label{fig:P-graph_critical}
\end{minipage}
\end{figure}

\begin{figure}[h]
\pgfdeclarelayer{background}
\pgfsetlayers{background,main}
\tikzstyle{vertex}=[circle,draw=black,line width=1.5pt,fill=black!25,minimum size=20pt,inner sep=0pt,font=\small]
\tikzstyle{selected vertex} = [vertex, fill=red!24]
\tikzstyle{edge} = [draw,line width=2pt,-]
\tikzstyle{weight} = [font=\small]
\tikzstyle{selected edge} = [draw,line width=5pt,-,red!50]
\tikzstyle{ignored edge} = [draw,line width=5pt,-,black!20]

\begin{minipage}[b][11cm][s]{.6\textwidth}
\vfill
\begin{tikzpicture}[scale=1, auto,swap]

\draw[fill=gray!25] (0,0)--(4,4)--(5,0);
\draw[fill=gray!25] (0,0)--(3,-3)--(5,0);

\foreach \pos/\name in {{(0,0)/x^2}, {(4,4)/y^4},{(3,-3)/z^3},{(5,0)/y^2z^2}}
        \node[vertex] (\name) at \pos {$\name$};

 \foreach \source/ \dest /\weight in {x^2/y^4/,x^2/z^3/,y^4/y^2z^2/,y^2z^2/z^3/,x^2/y^2z^2/
                                         }
 \path[edge] (\source) -- node[weight] {$\weight$} (\dest);

\end{tikzpicture}
\vfill
\caption{Reduced cell complex for the ideal $I=\langle x^2,y^4,y^2z^2,z^3\rangle$ .}\label{fig:P-graph_skeleton}
\end{minipage}
\end{figure}
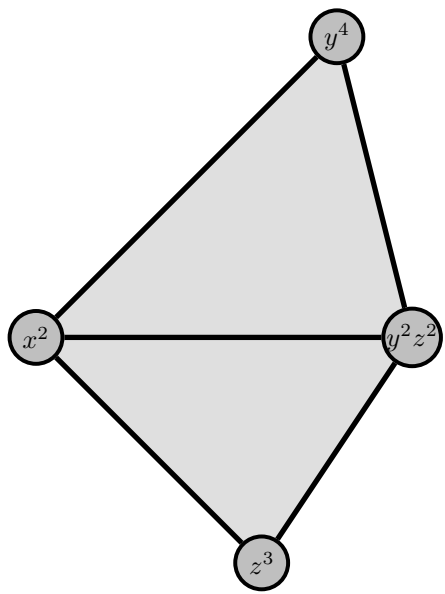
\end{example}

\bibliographystyle{plain}
 \bibliography{bibliography}
\end{document}